\newcommand{\E}[0]{{\sf E}}

\documentclass[letterpaper,11pt]{article}

\usepackage{amsthm}
\usepackage{amsmath}

\newtheorem{thm}{Theorem}[section]

\newtheorem{Lemma}[thm]{Lemma}

\newtheorem{lemma}[thm]{Lemma}
\newtheorem{prop}[thm]{Proposition}

\newtheorem{conj}[thm]{Conjecture}

\newtheorem{question}[thm]{Question}

\newcommand{\beq}[1]{\begin{equation}\label{#1}}
\newcommand{\enq}[0]{\end{equation}}

\newcommand{\bn}[0]{\bigskip\noindent}
\newcommand{\mn}[0]{\medskip\noindent}
\newcommand{\nin}[0]{\noindent}

\newcommand{\sub}[0]{\subseteq}
\newcommand{\sm}[0]{\setminus}

\newcommand{\ra}[0]{\rightarrow}
\newcommand{\Ra}[0]{\Rightarrow}

\newcommand{\ZZ}[0]{{\bf Z}}

\newcommand{\Nn}[0]{{\bf N}}

\newcommand{\bb}[0]{b}

\newcommand{\ttt}[0]{t}

\newcommand{\0}[0]{\emptyset}

\renewcommand{\qed}[0]{\begin{flushright} \rule{2mm}{3mm} \end{flushright}}

\newcommand{\C}[2]{{{#1}\choose{{#2}}}}
\newcommand{\Cc}[0]{\tbinom}
\newcommand{\ga}[0]{\alpha }

\newcommand{\gc}[0]{\gamma }

\newcommand{\gG}[0]{\Gamma }

\newcommand{\gl}[0]{\lambda }

\newcommand{\gS}[0]{\Sigma}
\newcommand{\gz}[0]{\zeta}
\newcommand{\eps}[0]{\varepsilon }
\newcommand{\vt}[0]{\vartheta}

\newcommand{\vp}[0]{\varphi}

\newcommand{\sugg}[1]{}

\newcommand{\comments}[1]{}
\usepackage[usenames]{color}
\begin{document}\renewcommand{\thefootnote}{\fnsymbol{footnote}}
\footnotetext{AMS 2010 subject classification:  05D40, 05C35, 05C80}
\footnotetext{Key words and phrases:  Mantel's Theorem,
random graph, threshold, max cut
}

\title{Mantel's Theorem for random graphs}
\author{B. DeMarco\footnotemark $~$ and J. Kahn\footnotemark
}
\date{}
\footnotetext{ * supported by the U.S.
Department of Homeland Security under Grant Award Number 2007-ST-104-000006.}
\footnotetext{ $\dag$ Supported by NSF grant DMS0701175.}

\date{}

\maketitle

\begin{abstract}
For a graph $G$, denote by $t(G)$ (resp. $b(G)$)
the maximum size of a triangle-free (resp. bipartite)
subgraph of $G$.
Of course $t(G) \geq b(G)$ for any $G$,
and a classic result of Mantel from 1907 (the first case of Tur\'an's Theorem)
says that equality holds for complete graphs.
A natural question, first considered by
Babai, Simonovits and Spencer about 20 years ago is,
when (i.e. for what $p=p(n)$) is the
``Erd\H{o}s-R\'enyi" random graph $G=G(n,p)$ likely to satisfy $t(G) = b(G)$?
We show that this is true if $p>C n^{-1/2} \log^{1/2}n $
for a suitable constant $C$,
which is best possible up to the value of $C$.
\end{abstract}

\section{Introduction}

\medskip
Write $\ttt(G)$ (resp. $b(G)$)
for the maximum size of a triangle-free
(resp. bipartite) subgraph of a graph $G$.
Of course $\ttt(G)\geq b(G)$,
and
Mantel's Theorem \cite{Mantel}
(the first case of Tur\'an's Theorem \cite{Turan})
says that equality holds if $G=K_n$.
Here we are interested in understanding
when
equality is likely
to hold for the usual (``Erd\H{o}s-R\'enyi") random graph
$G=G_{n,p}$; that is, for what $p=p(n)$
one has
\beq{tttbb}
\ttt(G_{n,p})=\bb(G_{n,p}) ~~~\mbox{{\em w.h.p.} }
\enq
(where an event holds {\em with high probability}
(w.h.p.) if its probability tends to 1 as $n\ra\infty$).
Note that \eqref{tttbb}
holds for very small $p$, for the silly reason
that $G$ is itself likely to be bipartite; but we are really
thinking of more interesting values of $p$.

The problem seems to have first been considered by
Babai, Simonovits and Spencer \cite{BSS},
who showed ({\em inter alia}) that \eqref{tttbb}
holds for $p>1/2$
(actually for $p>1/2-\eps$ for some
fixed $\eps >0$),
and asked whether it could be shown to hold for $p > n^{-c}$
for some fixed positive $c$.
This was accomplished by Brightwell, Panagiotou and Steger \cite{BPS}
(with $c = 1/250$),
who also suggested that $p> n^{-1/2+\eps}$ might be enough.
Here we prove the correct result and a little more:

\begin{thm}\label{MT}
There is a C such that if
$p>Cn^{-1/2}\log^{1/2} n$, then w.h.p. every maximum triangle-free
subgraph of $G_{n,p}$ is bipartite.
\end{thm}
\nin
This is best possible (up to the value of $C$),
since, as observed in \cite{BPS},
for
$p = 0.1 n^{-1/2}\log^{1/2}n$,
$G_{n,p}$ will usually contain
a 5-cycle of edges not lying in triangles.
In fact it's not hard to see that the probability in
\eqref{tttbb} tends to zero for, say,
$p \in [n^{-1},0.1 n^{-1/2}\log^{1/2}n]$, whereas,
as noted above, \eqref{tttbb} again holds for very small $p$.
An appealing guess is that, for a given $n$,
$f(p):=\Pr (\ttt(G_{n,p})=\bb(G_{n,p}))$
has just
one local minimum; but we have no idea how a proof of this
would go, or even any strong conviction that it's true.

\medskip
Of course a more general question is, what happens
when we replace ``triangle" by ``$K_r$" (and ``bipartite"
by ``$(r-1)$-partite")?
With $\ttt_r(G)$ (resp. $b_r(G)$)
the maximum size of a $K_r$-free
(resp. $(r-1)$-partite) subgraph of $G$, the natural extension
of Theorem \ref{MT} to general $r$ is
\begin{conj}\label{bpsconj}
For any fixed r there is a C such that if
$$ 
p > Cn^{-\tfrac{2}{r+1}}\log^{\tfrac{2}{(r+1)(r-2)}}n,
$$ 
then w.h.p. every maximum $K_r$-free subgraph of $G_{n,p}$ is $(r-1)$-partite.
\end{conj}
\sugg{\begin{conj}\label{bpsconj}
For any fixed r there is a C such that
\beq{tttbb'}
\Pr(\ttt_r(G_{n,p})=\bb_r(G_{n,p}))\ra 1 ~~~ (n\ra\infty)
\enq
provided
\beq{bpscond}
p > Cn^{-\tfrac{2}{r+1}}\log^{\tfrac{2}{(r+1)(r-2)}}n.
\enq
\end{conj}}
\nin
(This 
is again best possible
apart from the value of $C$, basically because for smaller $p$
there are edges not lying in $K_r$'s.)
The argument of \cite{BPS} gives
the conclusion of Conjecture \ref{bpsconj} 
provided $p> n^{-c_r}$
for a sufficiently small $c_r$.

\medskip
The next section states our two main points,
Lemmas \ref{HPilemma} and \ref{Pilemma},
and gives the easy derivation of Theorem \ref{MT}
from these.
The lemmas themselves are proved in Sections
\ref{PL1} and \ref{PL2}, following some routine
treatment of unlikely events in Section
\ref{Preliminaries}, and we close in Section \ref{Remarks}
with a few comments on related issues.

\mn
{\bf Usage.}
Throughout the paper we use $G$ for $G_{n,p}$ and $V$ for its
vertex set.
We use $|H|$ for the {\em size}, i.e. number of edges, of a graph $H$,
$N_H(x)$ for the set of neighbors
of $x$ in $H$, and $d_H(x)$ for the degree, $|N_H(x)|$, of $x$ in $H$.
The default value for $H$ is $G$;
thus
$N(x)=N_G(x)$ and, for $B\sub V$,
$d_B(x)=|N(x)\cap B|$ (and $d_B(x,y)=|N(x)\cap N(y)\cap B|$).
For disjoint $S,T\sub V$, the set of edges joining $S,T$
in $H$ is denoted
$\nabla (S,T)$ if
$H=G$ and $H[S,T]$ otherwise.

We will sometimes think of an $R\sub \C{V}{2}$
as the graph $(V(R),R)$, with
$V(R)$ the set of vertices contained in members of $R$;
so for example $N_R(x)$ is the set of $R$-neighbors of $x$,
$R[W]$ is the subgraph of $R$ induced by $W\sub V$,
and ``$R$ is bipartite" has the obvious meaning.

When speaking of a
{\em cut} $~\Pi=(A,B)$, we will think of $\Pi$ as either the
set of edges $\nabla(A,B)$
or as the
{\em ordered} partition $A\cup B$ of $V$
(so we distinguish $\Pi=(A,B)$ and $\Pi=(B,A)$).
Of course $|\Pi|$ means
$|\nabla(A,B)|$.

We use $\log$ for $\ln$,
$B(m,p)$ for a random variable with the binomial distribution
${\rm Bin}(m,p)$, and ``$a= (1\pm \vartheta)b$"
for ``$(1-\vt)b\leq a\leq (1+\vt)b$."
Following common practice, we usually
pretend that large numbers are integers,
to avoid cluttering the exposition with essentially irrelevant
floor and ceiling symbols.

\section{Outline}

We assume from now on that $p> Cn^{-1/2}\log^{1/2}n$ with
$C$ a suitably large constant
and
$n$ large enough to support the arguments below.
In slightly more detail: we fix small positive constants
$\eps$ and $\eta$ with $\eps>>\eta$,
set $\ga =.8$, and take $C$ large relative to $\eps$.
(The most stringent demand on $C$ is that it be somewhat large
compared to $\eps^{-5/2}$; see the end of Section \ref{PL1}.
For $\ga$, any value in $(2/3,1)$ would suffice.
Apart from this, we will mostly avoid numerical values:
no optimization is attempted, and
it will be clear in what follows that the constants can be chosen
to do what we ask of them.)

Say a cut $(A,B)$ is {\em balanced} if $|A|=(1\pm \eta)n/2$;
though we will sometimes speak more generally,
all cuts of actual interest below will be balanced.

We will need the following version of
a result of Kohayakawa, \L uczak and R\"odl
\cite{KLR}.
(See \cite[Theorem 8.34]{JLR}
and e.g. \cite[Proposition 1.12]{JLR}
for the standard fact that the $G_{n,M}$ statement
implies the $G_{n,p}$ version.)
\begin{thm}\label{8.34}
For each $\vartheta>0$ there is a $K$ such that
for $p =p(n)> Kn^{-1/2}$ w.h.p. each triangle-free
subgraph of $G=G_{n,p}$ of size at least $|G|/2$
can be made bipartite by deletion of at most $\vartheta n^2p$ edges.
\end{thm}
\nin
See Section \ref{Remarks} for a little more on Theorem \ref{8.34}.

For a cut $\Pi =(A,B)$,
set
$X(\Pi) = \{x\in A:d_B(x)<(1-2\eps)np/4\}$
and $T(\Pi)= \{x\in A: d_B(x)<(1-\eps)np/2\}$ ($\supseteq X(\Pi)$),
and let $Q(\Pi)$ consist of those pairs $\{x,y\}$ from $A$ which either
meet $X(\Pi)$ or satisfy one of

\mn
(i) $x,y\in A\sm T(\Pi)$
and $d_B(x,y)<\alpha np^2/2$;

\mn
(ii) $|\{x,y\}\cap T(\Pi)|=1$
and $d_B(x,y)<\alpha np^2/4$;

\mn
(iii) $\{x,y\}\sub T(\Pi)$
and $d_B(x,y)<\alpha np^2/8$.

\mn
In addition we
set
$Q_v(\Pi)= \{\{x,y\}\in Q(\Pi):\{x,y\}\cap X(\Pi)\neq 0\}$ and
$Q_e(\Pi)=Q(\Pi)\sm Q_v(\Pi)$.
Note that members of
$Q(\Pi)$,
while often treated as edges of an auxiliary graph,
need not be edges of $G$.

\medskip
For a cut $\Pi =(A,B)$ and $F\sub G$, let
$$\vp(F,\Pi) = 2|F[A]| +|F[A,B]|. $$

\begin{lemma}\label{HPilemma}
W.h.p.
$$\vp(F,\Pi) < |\Pi|$$
whenever $\Pi =(A,B)$
is balanced
and $F\sub G$ is
triangle-free with $F\neq \Pi$, $F\cap Q(\Pi)=\0=F[B]$,
\beq{HP2}
|F[A]|<\eta |F[A,B]|,
\enq
and
\beq{HP3}|N_F(x)\cap B|\geq |N_F(x)\cap A| ~~\forall x\in A.
\enq
\end{lemma}

\begin{lemma}\label{Pilemma}
W.h.p.
\beq{ttt1}
\bb(G)> |\Pi| +2|Q|
\enq
whenever the balanced cut $\Pi=(A,B)$ and $\0\neq Q\sub G\cap Q(\Pi)$ satisfy
\beq{dbad}
d_Q(x)\leq d_B(x) ~~\forall x\in A.
\enq
\end{lemma}

\bigskip
Given Lemmas \ref{HPilemma} and \ref{Pilemma} we finish easily
as follows.
Let $F_0$ be a maximum triangle-free subgraph of $G$,
and $\Pi=(A,B)$ a cut maximizing $|F_0[A,B]|$ with
(w.l.o.g.) $|F_0[A]|\geq |F_0[B]|$.
Since $\Pi$ maximizes $|F_0[A,B]|$, we have \eqref{HP3}
(with $F_0$ in place of $F$)---otherwise we could move some
$x$ from $A$ to $B$ to increase $|F_0[A,B]|$---and
Theorem \ref{8.34} implies that w.h.p.
$F_0$ also satisfies \eqref{HP2} (actually with $o(1)$ in place of $\eta$).
Moreover $\Pi$ is balanced w.h.p., since (w.h.p.)
\begin{eqnarray}
|\nabla(A,B)|&\geq &|F_0[A,B]|>(1-o(1))|F_0| \nonumber\\
&\geq &(1-o(1))|G|/2
> (1-o(1))n^2p/4\label{nabAB1}
\end{eqnarray}
and, for example,
\beq{nabAB2}
|\nabla(A,B)| <\left\{\begin{array}{ll}
(1+o(1))|A||B|p &\mbox{if $|A|,|B| > n/5$}\\
(1+o(1))\min\{|A|,|B|\}np &\mbox{otherwise.}
\end{array}\right.
\enq
Here the second inequality in \eqref{nabAB1}
is again Theorem \ref{8.34},
and the third is the standard observation that
$b(G)\geq |G|/2$ for any $G$.
The last inequality in \eqref{nabAB1} and those in
\eqref{nabAB2} are easy
consequences of
Chernoff's inequality (Theorem \ref{Chern} below,
used {\em via} Proposition \ref{vdegree} for the second inequality in \eqref{nabAB2}).

\medskip
Let
$F_1 =F_0\sm F_0[B]$ and $F=F_1\sm Q(\Pi)$.
Noting that these modifications introduce no triangles
and preserve
\eqref{HP2} and \eqref{HP3},
we have, w.h.p.,
\begin{eqnarray}
\ttt(G) &=&|F_0|\nonumber\\
& \leq & \vp(F_1,\Pi)\nonumber\\
& =& \vp(F,\Pi) + 2|F_1\cap Q(\Pi)|\nonumber\\
&\leq & |\Pi| + 2|F_1\cap Q(\Pi)|\label{final1}\\
&\leq &\bb(G).\label{final2}
\end{eqnarray}
Here \eqref{final1} is given by Lemma \ref{HPilemma}
and \eqref{final2} by Lemma \ref{Pilemma}
(the latter applied with $Q=F_1\cap Q(\Pi)$ and
\eqref{dbad} implied by \eqref{HP3} for $F_1$).

This gives
\eqref{tttbb}.
For the slightly stronger assertion in the theorem,
notice that we have strict inequality in
\eqref{final1} unless $F=\Pi$
and in \eqref{final2} unless $F_1\cap Q(\Pi)=\0$.
Thus $|F_0|=b(G)$ implies
$F_0[A]=F[A]\cup (F_1\cap Q(\Pi))=\0$,
so also $F_0[B]=\0$
(since we assume $|F_0[A]|\geq |F_0[B]|$).
\qed

\section{Preliminaries}\label{Preliminaries}
Here we just dispose of some anomalous events.
We use Chernoff's inequality in the following form, taken from
\cite[Theorem 2.1]{JLR}.
\begin{thm}\label{Chern}
For $\xi =B(n,p)$, $\mu=np$ and any $\gl\geq 0$,
\begin{eqnarray*}
\Pr(\xi \geq \mu+\gl)& < &\exp [- \tfrac{\gl^2}{2(\mu+\gl/3)}],\\
\Pr(\xi \leq \mu-\gl) &<& \exp [- \tfrac{\gl^2}{2\mu}].
\end{eqnarray*}
\end{thm}

\nin
This easily implies the next two standard facts, whose proofs we omit.
\begin{prop}\label{vdegree}
W.h.p. for all $x,y\in V$,
\beq{Deg}
d(x)=(1\pm o(1))np ~~~\mbox{and}
~~~d(x,y)=(1\pm \eps)np^2,
\enq
\end{prop}

\begin{prop}\label{density}
There is a $K$
such that w.h.p.
for all disjoint $S,T\sub V$ of size at least
$Kp^{-1}\log n$,
\beq{NST}
|\nabla(S,T)|=(1\pm \eps) |S||T|p
\enq
and
\beq{NST'}
|G[S]| =(1\pm \eps) \Cc{|S|}{2}p.
\enq
\end{prop}

\nin
The next three assertions are also easy consequences of
Theorem \ref{Chern}.
\begin{prop}\label{someversion}
There is a K such that
w.h.p., for every $\kappa  > K p^{-1}\log n$
and
$S,T\neq \0$ disjoint subsets of $V$ with $|S|\leq \min\{\kappa,|T|\}$,
\beq{chernoff1}
|\nabla(S,T)| \leq 2|T|\kappa p
\enq
and
\beq{chernoff2}
|G[S]|\leq  |S|  \kappa p.
\enq
\end{prop}

\nin
{\em Proof.}
We show \eqref{chernoff1}, omitting the similar proof of
\eqref{chernoff2}.
For given $s,t$ with $s\leq t$, the number of possibilities for
$S$ and $T$
of sizes $s$ and $t$ respectively ($s\leq t$) is less than
$\binom{n}{s}\binom{n}{t} <\exp [2t\log n]$.
But for a given $S,T$, since
$\E |\nabla(S,T)| = |S||T|p \leq |T| \kappa p,$ Theorem \ref{Chern}
gives (say)
$$
\Pr(|\nabla(S,T)|\geq 2|T| \kappa p) < \exp[-|T|\kappa p/3].
$$
The probability that \eqref{chernoff1} fails for some $\kappa,S,T$
is thus
at most
$$\mbox{$n^2\sum_{t>0}\exp[(2-K/3)t\log n]$}$$
(where the $n^2$ covers choices for $s,\kappa\in [n]$),
which is $o(1)$ if $K > 12$.
\qed

\begin{prop}\label{PiXprop}
There is a K such that w.h.p. $|T(\Pi)|< Kp^{-1}$
for every balanced cut $\Pi$.
\end{prop}

\nin
{\em Proof.}
The number of possibilities for $\Pi=(A,B)$ and
a $T\sub A$ of size $t:=\lceil K/p\rceil$ is less than
$\exp_2[n +t \log_2 n]$, while
for such a $\Pi$ and $T$,
$$
\Pr(T(\Pi)\supseteq T) <
\Pr(|\nabla(T,B)| < (1-\eps)tnp/2)<\exp[-ctnp],
$$
with $c\approx \eps^2/4$
(using $|B|> (1-\eta)n/2$).  The proposition follows,
e.g. with $K=4\eps^{-2}$.\qed

\begin{prop}\label{ldegreeQ}
There is a $K$ such that w.h.p. for every
cut $\Pi =(A,B)$ and $x\in A\sm X(\Pi)$,
\beq{ldegree}
d_{ Q_e(\Pi)}(x) < K/p .
\enq
\end{prop}

\nin
{\em Proof.}
By Proposition \ref{vdegree} it's enough to show that w.h.p.
\eqref{ldegree} holds whenever (say)
$d(x)\leq(1+\eps)np$.
Noting that a
violation at $x$ (and some $\Pi$)
implies that there are disjoint
$S\sub V$ and $T\subseteq N(x)$ with
$|T|\geq t:=(1-2\eps)np/4$,
$|S|=s:= \lceil K/p\rceil$
and $|\nabla(S,T)|<\frac{\alpha}{1-2\eps}s|T|p=:(1-\gz)s|T|p$,
we find that
the probability of such a violation with
$d(x)\leq(1+\eps)np$
is at most
$$n^22^{(1+\eps)np}\Cc{n}{s}\exp[-\gz^2stp/2],$$
which
is $o(1)$ for sufficiently large $K$
(e.g. $K=5000$ is enough).\qed

\section{Proof of Lemma \ref{HPilemma}}\label{PL1}
We will show that the ``w.h.p." statement in
Lemma \ref{HPilemma} holds whenever we have the conclusions of
Propositions \ref{vdegree}-\ref{PiXprop};
so we assume in this section that these conclusions hold for
$K$, which we take to be
the largest of the $K$'s appearing in these propositions
(so $K\approx 4\eps^{-2}$, which is what's needed in
Propositions \ref{density} and \ref{PiXprop}).

To keep the notation simple, we set, for a given $\Pi=(A,B)$ and $F$,
$$I= F[A], ~J=F[A,B], ~L = G[A,B]\sm J,$$
and write, e.g., $I(x)$ for the set of edges of $I$
containing $x$.

%
We may 
assume that, given $\Pi$, $F$
maximizes $\vp(F,\Pi)$ subject to the conditions of the lemma.
Notice that this implies
\beq{IJ}
d_I(x)\geq d_L(x)/2 ~~\forall x\in A,
\enq
since if $x$ violates \eqref{IJ} then
$F':=(F\sm I(x))\cup L(x)$ satisfies
the conditions of the lemma (using $F[B]=\0$ to say $F'$ is triangle-free)
and has
$\vp(F',\Pi)>\vp(F,\Pi)$.  We will actually show that if \eqref{IJ}
is added to our other
assumptions then $I=\0$, whence $F\subset \Pi$ and $\vp(F,\Pi)=|F|<|\Pi|$;
so we now assume \eqref{IJ}.

\medskip
Set
$T=T(\Pi)\sm X(\Pi)$,
$S=\{x\in A\sm  T: d_I(x)>\eps np\}$,
$R=A\sm(S \cup T)$,
$T_1=\{x\in T: d_I(x)>\eps np\}$ and
$T_2= T\sm T_1$.
Let
$$
M = |\{(x,y,z):xy\in I, xz\in L, yz\in G\}|.
$$
(Note $xy\in I\Ra x,y \in A$ and then  $xz\in L\Ra z\in B$.)
Since $F$ is triangle-free, we have
\beq{M}
\sum_{x\in A}|\nabla(N_I(x),N_L(x))| ~ =~ M ~
\geq~
\sum_{x\in A}|\nabla(N_I(x),N_J(x))|.
\enq
So if we set
$g(x) =|\nabla(N_I(x),N_L(x))|$ and
$f(x) = |\nabla(N_I(x),N_J(x))|$ (for $x\in A$), then
\eqref{M} says
$$
\sum_{x\in A}(g(x)-f(x))\geq 0,
$$
whereas we'll show
\beq{contradiction}
\mbox{$\sum_{x\in A}(g(x)-f(x))< 0~$ unless $~I=\0$.}
\enq

\mn
{\em Proof.}
We first assert that
\beq{upshot1}
g(x)-f(x)<\left\{\begin{array}{ll}
(1+4\eps)d_I(x)np^2/3& \mbox{if $x\in S$,}\\
(1+4\eps)d_I(x) np^2/6& \mbox{if $x\in T_1$.}
\end{array}\right.
\enq
To see this, rewrite
\beq{gfnabla}
g(x)-f(x)=
|\nabla(N_I(x),N_B(x))|
-2|\nabla(N_I(x),N_J(x))|.
\enq
For $x\in S\cup T_1$, \eqref{NST}
(with \eqref{HP3}) gives
$|\nabla(N_I(x),N_B(x))|<(1+\eps)pd_I(x)d_B(x)$
and
$|\nabla(N_I(x),N_J(x))|>(1-\eps)pd_I(x)d_J(x)$,
while
$$
d_J(x) \geq d_B(x)/3
$$
for any $x\in A$
(since
$d_L(x)+d_J(x)=d_B(x)~$
and, according to \eqref{IJ} and \eqref{HP3},
$~d_L(x)\leq 2d_I(x) \leq 2d_J(x)$).
Inserting these bounds in \eqref{gfnabla} and using
(quite unnecessarily)
$d_B(x)\leq d(x)-d_I(x)<(1+o(1)-\eps)np$
(see \eqref{Deg}) gives \eqref{upshot1}.

\medskip
We next consider $x\in R\cup T_2$, and rewrite
\beq{gyfy}
g(x)-f(x) = 2|\nabla(N_I(x),N_L(x))| - |\nabla(N_I(x),N_B(x))|.
\enq
We consider the two terms on the right separately,
beginning with the second.
Recalling that $I\cap Q(\Pi)=\0$
and setting $d_I'(x)=|N_I(x)\sm T|$,
$d_I''(x)=|N_I(x)\cap T|$,
we have
\beq{second1}
|\nabla(N_I(x),N_B(x))|
\geq \left \{\begin{array}{ll}
\ga np^2 \left( d_I'(x)/2  + d_I''(x) /4  \right)
&\text{if $x\in R$,}\\
\ga np^2 d_I(x) /8 &
\text{if $x\in T_2$.}
\end{array}\right.
\enq

For the first term on the r.h.s. of \eqref{gyfy} we have
\beq{third}
x\in R\cup T_2  ~~ \Ra ~~
|\nabla(N_I(x),N_L(x))| ~<~ d_I(x)\cdot 4\eps np^2,
\enq
using \eqref{chernoff1}
and
the fact that $x\in R\cup T_2$ implies
$d_L(x)\leq 2d_I(x)\leq 2\eps np.$
(In more detail:  if $d_L(x)\leq d_I(x)$ then we use
\eqref{chernoff1}
with $S=N_L(x)$, $T=N_I(x)$ and $\kappa = \eps np$; otherwise,
we take $S=N_I(x)$, $T=N_L(x)$ and $\kappa = \eps np$
to obtain the bound $d_L(x)\cdot 2\eps np^2
\leq d_I(x)\cdot 4\eps np^2$.)

In particular, for $x\in T_2$ we have
\beq{T2}
g(x)-f(x)
~\leq ~d_I(x)np^2 (8\eps-\ga/8)
~\leq ~0.
\enq

\medskip
Collecting the information from
\eqref{upshot1} and
\eqref{gyfy}-\eqref{T2}, we find that the sum
in \eqref{contradiction}
is bounded above by

\beq{bd0}
np^2 [(1+4\eps)(\sum_{x\in S} \tfrac{d_I(x)}{3} +
\sum_{x\in T_1} \tfrac{d_I(x)}{6}) +
\sum_{x\in R}
\{8\eps d_I(x)- \ga (\tfrac{d_I'(x)}{2}  + \tfrac{d_I''(x)}{4})\}].
\enq
%
So we just need to show that this is negative if $I\neq \0$,
which follows from
\beq{bd1}
\sum_{x\in R} d_I' (x) \geq .9 \sum_{x\in S} d_I(x)
\enq
and
\beq{bd2}
\sum_{x\in R} d_I'' (x) \geq .9 \sum_{x\in T_1} d_I(x)
~~\mbox{if (say)
$|T_1|\geq \eta|S|$.}
\enq
(If $\eta |S|>|T_1|$ then
\eqref{bd1} is enough and we don't need the $d_I''$ terms in
\eqref{bd0}.)

\mn
The proofs of \eqref{bd1} and \eqref{bd2} are
similar and we just give the first.

\mn
{\em Proof of \eqref{bd1}.}
We may of course assume $S\neq\0$.
Since $\sum_{x\in S} d_I(x) = \sum_{x\in R}d_I'(x) +
|I[S,T]|+2|I[S]|$,
it's enough to show
\beq{bd1toshow}
\mbox{$|\nabla(S,T)|+2|G[S]|\leq .1 \sum_{x\in S} d_I(x).$}
\enq
Notice that
$|T|<Kp^{-1}$ (see Proposition \ref{PiXprop})
and $|S| < (\eta/\eps)n$ (by \eqref{HP2} and \eqref{NST},
the latter applied to $|\nabla(A,B)|$).
Combining these bounds with the conclusions of
Proposition \ref{someversion} (using $\kappa = (\eta/\eps)n$
and $\kappa = Kp^{-1}\log n$ respectively)
gives
$|G[S]|\leq |S| (\eta/\eps)np$ and
$$
|\nabla(S,T)|\leq 2\max\{|S|,|T|\}K \log n
\leq \left\{\begin{array}{ll}
2K|S|\log n&\mbox{if $|S|\geq |T|$,}\\
2K^2p^{-1}\log n&\mbox{if $|S|< |T|$.}
\end{array}\right.
$$
So, noting that $\sum_{x\in S} d_I(x)>|S|\eps np$
and that $2K^2p^{-1}\log n$ is small relative to $\eps np$,
we have \eqref{bd1toshow}.\qed

\section{Proof of Lemma \ref{Pilemma}}\label{PL2}

For $\Pi=(A,B)$, let  $\Pi^*= (A\sm X(\Pi), B\cup X(\Pi))$.
Propositions \ref{vdegree} and \ref{PiXprop} imply that w.h.p.
\begin{eqnarray}
|\Pi^*|&\geq &|\Pi| + \sum_{x\in X(\Pi)} (d(x)-2d_B(x)-|X(\Pi)|)\nonumber\\
&\geq &|\Pi| + |X(\Pi)|np/2 ~~~~~
\mbox{{\em for every balanced} $\Pi$}.
\label{PiPi*}
\end{eqnarray}
If $Q$ and $\Pi$ are as in Lemma \ref{Pilemma} and
$|Q\cap Q_v(\Pi)|>(1+\eps)^{-1}|Q|$ then, since
$|Q\cap Q_v(\Pi)|<|X(\Pi)|(1-2\eps)np/4$
(by \eqref{dbad}),
we have
$$|Q|<(1-\eps-2\eps^2)|X(\Pi)|np/4,$$
and it follows that
$$b(G)\geq |\Pi^*| > |\Pi| + 2|Q|$$
provided \eqref{PiPi*} holds.
It is thus enough to show
\begin{Lemma}\label{badedgeQ}
There is a $\delta>0$ such that w.h.p.
$$b(G)> |\Pi| + |Q|\delta np^2$$
for every balanced cut  $\Pi$ and $\0\neq Q\subseteq Q_e(\Pi)$.
\end{Lemma}

\begin{proof}
Set $\gc=(1-2\eps)/4$, $\ga'=\ga/(1-2\eps)$,
and let $\gz$ be a positive constant satisfying
$$\vt :=.9 -2\gz/\gc-\ga'>0.$$

By Proposition \ref{ldegreeQ} it's enough to
prove Lemma \ref{badedgeQ} when
$d_Q(x)<K/p$ ($K$ as in the proposition)
for all $x\in A$.
It's also easy to see that for any
such $Q$ and $\tau\in [ p,K]$, there is
a bipartite $R\sub Q$ with
\beq{R1}
d_R(x)\leq \lceil\tau/p\rceil ~~~\forall x
\enq
and
$
|R|\geq  \frac{\tau}{2K}|Q|.
$ 
(To see this, start with a bipartite $Q'\sub Q$ with $|Q'|\geq |Q|/2$.
Assigning each edge of $Q'$ weight $\tau/K$ gives total weight at
each vertex at most $\lceil\tau/p\rceil $ (actually $\tau/p$ of course,
but we want integers), and the Max-flow Min-cut Theorem then
gives the desired $R$.)
It thus suffices to prove Lemma \ref{badedgeQ}
with $Q$ replaced by a bipartite
$R\sub Q_e(\Pi)$
satisfying \eqref{R1}, where we set
$$\tau =\max\{\gz, p\}.$$

\mn
(We could of course just invoke \cite{BPS} to handle large $p$,
but it seems silly to avoid the few extra lines
needed to deal with this easier case.)

\medskip
For $X,Y$ disjoint subsets of $V$,
$f:X\ra \{k\in \Nn: k\geq \gc np\}$, and
$R\sub \{\{x,y\}:x\in X,y\in Y\}$
satisfying \eqref{R1}
with $V(R)=X\cup Y$,
denote by $E(R,X,Y,f)$ the event
that there is a balanced cut $\gS=(A,B)$ with
$R\sub Q_e(\gS)$,
\beq{RPi}
d_B(x) = f(x) ~~\forall x\in X,
\enq
and
$$
\bb(G)< |\gS| +\vt |R|\gc np^2.
$$

\mn
We will show
\beq{ERXYf}
\Pr(E(R,X,Y,f)) <\exp[-.001|R| np^2].
\enq
This is enough to prove Lemma \ref{badedgeQ}
(with $R$ in place of $Q$ as discussed above),
since the number of possibilities for $(R,X,Y,f)$
with $|R| =t$ is less than
$\C{\C{n}{2}}{t}2^tn^t < \exp [3t\log n]$.

\medskip
For the proof of \eqref{ERXYf} we think of choosing $G$
in stages:

\mn
(i)  Choose all edges of $G$ except those in $\nabla(Y,V\sm X)$.

\mn
(ii) Choose all remaining edges of $G$ except those belonging to the
sets $\nabla(y,\cup _{xy\in R}N_x)$ for $y\in Y$.

\mn
(iii)  Choose the remaining edges of $G$.

\mn
Let $G'$ be the subgraph of $G$ consisting of the edges chosen
in (i) and (ii),
and let $(S,T)$ be a balanced cut of $G'$ of maximum
size among those
satisfying
$$
\mbox{$X\cup Y\sub S~$ and $~d_T(x)=f(x) ~~\forall x\in X$.}
$$
(Of course if there is no such cut then $E(R,X,Y,f)$ does not occur.)
For each $y\in Y$ set $M(y)=(\cup _{xy\in R}N_x)\cap T$
and $F(y) =\sum_{xy\in R} f(x)$.

If $\gS=(A,B)$ and $R\sub Q_e(\gS)$, then
\beq{alpha'}
d_B(x,y) < \ga' \min\{d_B(x), d_B(y)\}p  ~~~\forall \{x,y\}\in R.
\enq
Our choice of $(S,T)$ gives
\begin{eqnarray}
|\gS|
&\leq &|G'[S,T]| + \sum_{y\in Y}\sum_{xy\in R}d_B(x,y)\nonumber\\
&\leq& |G'[S,T]| + \sum_{y\in Y}\sum_{xy\in R}\ga' d_B(x)p\nonumber\\
&=& |G'[S,T]| + \ga' p\sum_{y\in Y}F(y)\label{apF}
\end{eqnarray}
for any balanced cut $\gS=(A,B)$ satisfying
$X\cup Y\sub A$, \eqref{RPi} and \eqref{alpha'}, while
\beq{bGforY}
\bb(G)\geq |G'[S,T]| + \sum_{y\in Y}|\nabla(y,M(y))|.
\enq

Suppose
first that we are in the (main) case $p<\gz$ (so $\tau = \gz$).
Then w.h.p. (depending only on $G'$) we have
\begin{eqnarray}\label{anyY}
|M(y)|&\geq &\sum_{xy\in R}[d_T(x) - \sum\{d_T(x,x'):x\neq x'\in N_R(y)\}]\\
&> &(1-2\gz/\gc) F(y),
\nonumber
\end{eqnarray}
for each $y\in Y$,
since for any $x\in X$, $d_T(x)=f(x)\geq \gc np$
and the inner sum in \eqref{anyY} is (w.h.p.) at most
$$d_R(y) \max\{d(x,x'):x,x'\in V\} < \lceil\gz/ p\rceil (1+o(1))np^2
< 2\gz np.$$
So w.h.p. the sum in \eqref{bGforY} has the distribution
${\rm Bin}(m,p)$ for some
\beq{m}
m> (1-2\gz/\gc) \sum_{y\in Y}F(y) \geq (1-2\gz/\gc) |R|\gc np,
\enq
and exceeds $.9 mp$  with probability at least
$1-e^{-.005mp}> 1-e^{-.001|R| np^2}$;
and whenever this happens,
$\bb(G)$ exceeds the r.h.s. of \eqref{apF} by at least
\beq{last}
(.9-2\gz/\gc -\ga')\sum_{y\in Y}F(y)p \geq \vt |R| \gc np^2.
\enq

If $p\geq \gz$, then $R$ is a matching and the inner sum in \eqref{anyY}
is empty.  So
we have $m\geq |R|\gc np$ in \eqref{m},
and in \eqref{last} can replace $.9-2\gz/\gc -\ga'$ by
$.9-\ga'$.
\end{proof}

\section{Remarks}\label{Remarks}
We continue to write $G $ for $G_{n,p}$.

Of course the next goal is to prove Conjecture 1.2.
At this writing we think we may know how to do this,
but the argument envisioned goes well beyond present ideas
and, if correct, will appear separately.

\bn
{\bf Theorem \ref{8.34}?}
It would be interesting to know whether
Theorem \ref{MT} can be proved without Theorem \ref{8.34}.
This is not to say that such a proof would necessarily help
in proving Conjecture \ref{bpsconj},
but, consequences aside, it seems interesting to understand
whether this relatively
difficult ingredient is really needed, or is
just a convenience.
(The extension of Theorem \ref{8.34} to larger $r$,
suggested in \cite{Koh,KLR}, was achieved by Conlon and Gowers in
\cite{Conlon-Gowers} and given a different proof,
building on work of Schacht \cite{Schacht}, by Samotij in \cite{Samotij}.)

It has also seemed interesting to give an easier
proof of Theorem \ref{8.34}.
The original proof \cite{KLR, JLR} uses a
sparse version of Szemer\'edi's Regularity Lemma \cite{Szemeredi}
due to Kohayakawa \cite{Koh} and R\"odl (unpublished; see \cite{Koh}),
together with the triangle case of the
``KLR Conjecture" of \cite{KLR}.
The proofs of \cite{Conlon-Gowers} and
\cite{Samotij} avoid these tools (they do use the
``graph removal lemma"
of \cite{EFR}---so for Theorem \ref{8.34} itself the
``triangle removal lemma" of Ruzsa and Szemer\'edi
\cite{Ruzsa-Szemeredi}---but
not in an essential way \cite{Conlon}),
but
are rather difficult.
Here we just mention that we do now know a reasonably simple proof
of
Theorem \ref{8.34}.
This argument
will appear
in \cite{DHK}.

\bn
{\bf Homology.}
It's not too hard to show that (roughly speaking)
if $p$ is as in Theorem \ref{8.34}, then w.h.p.
every triangle-free $F\sub G$
with $|F|\geq |G|/2$ has even intersection with most triangles of $G$.
(This is essentially due to Frankl and R\"odl \cite{FR86},
following an idea of Goodman \cite{Goodman};
see also \cite[Sec.8.2]{JLR}.)
Thus in thinking about a new proof of Theorem \ref{8.34},
we wondered whether
some insight might be gained by understanding what happens
when one replaces ``most" by ``all."
This turns out not to be a new question:

Recall that the {\em clique complex}, $X(H)$, of a graph $H$ is the
simplicial complex whose faces are the (vertex sets of)
cliques of $H$.  (For background on this and related topological
notions, see for example \cite{Roy,Kahle}.)
A precise conjecture, proposed by M. Kahle (\cite{KahlePC}; see also
\cite{Kahle}) and proved
by him for $\gG={\bf Q}$, is
\begin{conj}\label{CKahle}
Let $\gG$ be either $\ZZ$ or a field.
For each positive integer k and $\eps >0$,
if
$$
p>(1+\eps)\left[(1+k/2)(\log n/n)\right]^{1/(k+1)},
$$
then w.h.p. $H_k(X(G),\gG)=0$ (where $H_k$ denotes $k$th homology group).
\end{conj}
\nin
(For $k=0$---with, of course,
$H$ replaced by the reduced homology $\tilde{H}$---this is a classical
result of Erd\H{o}s and R\'enyi \cite{ER1} on
connectivity of $G_{n,p}$.)
The answer to the above question
(on edge sets having even intersection with all triangles) is
the case $k=1$, $\gG=\ZZ_2$ of Kahle's conjecture,
and an easy consequence of
the following precise statement (again taken from
\cite{DHK}), in which we take $Q$ to be the event that
every edge of $G$ is in a triangle.

\begin{thm} \label{topoThm}
For any $p=p(n)$
$$\Pr( Q \wedge H_1(X(G),\ZZ_2)\neq 0)\ra 0 ~~(n\ra\infty).$$
\end{thm}
\nin
In other words, w.h.p. either $Q$ fails or the only subsets of
$E(G)$ meeting each triangle an even number of
times are the cuts.
%
As it turned out, our original proof of this was {\em based on}
Theorem \ref{8.34}, so was not all that helpful from the point
of view mentioned above;
but we do now know how to show it without
using Theorem \ref{8.34}, and this was indeed
helpful in suggesting a new proof for the latter.

\bn
Department of Mathematics\\
Rutgers University\\
Piscataway NJ 08854\\
rdemarco@math.rutgers.edu\\
jkahn@math.rutgers.edu

\end{document}